\documentclass[oneside,notitlepage,12pt]{article}

\usepackage{amssymb}
\usepackage[leqno]{amsmath}
\usepackage{amsfonts}
\usepackage{latexsym}
\usepackage{amsopn}
\usepackage{amstext}
\usepackage{amsthm}

\usepackage[all]{xy}
\newdir{ >}{{}*!/-9pt/@{>}}

\usepackage{verbatim}
\usepackage[colorlinks]{hyperref}

\providecommand{\cal}{\mathcal}
\renewcommand{\Bbb}{\mathbb}

\newenvironment{pf}{\begin{proof}}{\end{proof}}

\newcommand{\Aaa}{{\cal{A}}}
\newcommand{\Bee}{{\cal{B}}}
\newcommand{\Cee}{{\cal{C}}}

\newcommand{\Ef}{{\cal{F}}}
\newcommand{\Gee}{{\cal{G}}}

\newcommand{\Pee}{{\cal{P}}}
\newcommand{\Raa}{{\cal{R}}}

\newcommand{\Nat}{{\Bbb{N}}}
\newcommand{\Err}{{\Bbb{R}}}

\newcommand{\lam}{{\lambda}}
\newcommand{\al}{\alpha}

\renewcommand{\phi}{\varphi}
\renewcommand{\rho}{\varrho}

\newcommand{\loe}{\leqslant}
\newcommand{\goe}{\geqslant}

\newcommand{\subs}{\subseteq}

\newcommand{\nnempty}{\ne\emptyset}

\newcommand{\Int}{\operatorname{int}}
\newcommand{\w}{\operatorname{w}}

\newcommand{\conv}{\operatorname{conv}}

\newcommand{\cf}{\operatorname{cf}}

\newcommand{\liminv}{\varprojlim}
\newcommand{\twoh}{\twoheadrightarrow}

\newcommand{\oraz}{\qquad\text{and}\qquad}

\newcommand{\Es}{{\cal{S}}}

\newcommand{\by}{/}

\newtheorem{tw}{Theorem}[section]
\newtheorem{wn}[tw]{Corollary}
\newtheorem{lm}[tw]{Lemma}
\newtheorem{prop}[tw]{Proposition}
\newtheorem{claim}[tw]{Claim}

\theoremstyle{definition}

\theoremstyle{remark}

\newcommand{\setof}[2]{\{#1\colon #2\}}

\newcommand{\seq}[1]{\langle #1 \rangle}

\newcommand{\sett}[2]{\{#1\}_{#2}}
\newcommand{\pair}[2]{\langle #1, #2 \rangle} 
\newcommand{\map}[3]{#1\colon #2 \to #3} 

\newcommand{\fra}{Fra\"iss\'e}
\newcommand{\jon}{J\'onsson}
\newcommand{\frajon}{\fra-\jon}

\newcommand{\fK}{{\mathfrak{K}}}
\newcommand{\fL}{{\mathfrak{L}}}

\newcommand{\cmp}{\circ} 

\newcommand{\komp}{\ensuremath{\mathfrak C\mathfrak o\mathfrak m\mathfrak p}} 

\newcommand{\wek}[1]{{\vec{#1}}}

\newcommand{\parov}{Paro\-vi\-\v{c}en\-ko}

\newcommand{\qmap}[3]{#1\colon #2 \twoheadrightarrow #3} 

\title{\parov\ spaces with structures}

\author{
{\sc Wies{\l}aw Kubi\'s}
\footnote{Research supported by the GA\v{C}R grant P
201/12/0290 (Czech Republic)}\\
{\small Mathematical Institute, Academy of Sciences of the Czech Republic}\\
{\small Prague, CZECH REPUBLIC}\\
{\small\texttt{kubis@math.cas.cz}}
\and
{\sc Andrzej Kucharski}\\
{\small Institute of Mathematics, University of Silesia}\\
{\small Katowice, POLAND}\\
{\small\texttt{akuchar@math.us.edu.pl}}
\and
{\sc S\l awomir Turek}\\
{\small Institute of Mathematics, Jan Kochanowski University}\\
{\small Kielce, POLAND}\\
{\small\texttt{sturek@ujk.edu.pl}}
}

\begin{document}

\maketitle

\begin{abstract}
We study an analogue of the \parov\ property in categories of compact spaces with additional structures.
In particular, we present an internal characterization of this property in the class of compact median spaces.
\end{abstract}

\vspace{3mm}


\tableofcontents

\section{Introduction}

A compact space $K$ is said to be \emph{\parov} if for every continuous surjections $\map f AB$ and $\map g KB$, where $A$, $B$ are compact metrizable spaces, there exists a continuous surjection $\map h K A$ satisfying $f \cmp h = g$.
This definition naturally generalizes to higher weights, leading to the notion of a $\kappa$-\parov\ space, where $\kappa$ is a cardinal bounding the weights of the spaces $A$, $B$.
It is also natural to consider the same property in other categories.
We shall concentrate on categories of \emph{structured compact spaces}, defined below.
We shall present a general (and rather standard) result on the existence of structured \parov\ spaces.
Next, we discuss one particular example: compact median spaces.
The main result is an internal characterization of \parov\ objects in this category.

The concept of a \parov\ space was introduced by van Douwen and van Mill~\cite{vMillvDouwen}, where the authors actually used its internal properties, instead of the external category-theoretic condition stated above.
Originally, \parov~\cite{Parov} showed that every compact space of weight $\loe \aleph_1$ is a continuous image of the \v Cech-Stone remainder $\Nat^*$ of the natural numbers; he also proved
 that, assuming the Continuum Hypothesis, this is the unique \parov\ compact space of weight $\mathfrak c$.

Later on, Negrepontis~\cite{Negre} generalized \parov's results to higher cardinals; some of his questions were several years later answered by Dow~\cite{AD}.
Finally, van Douwen and van Mill~\cite{vDvM} showed that the Continuum Hypothesis is equivalent to the existence of a unique \parov\ space of weight $\aleph_1$.
At the same time, B\l aszczyk and Szyma\'nski~\cite{BlSz} gave a simpler proof of \parov\ theorem, also finding another characterization of $\Nat^*$, under the Continuum Hypothesis.
Some further simplifications were subsequently found by Engelking~\cite{Engelking}.
Nowadays, \parov\ spaces (dually, \parov\ Boolean algebras) belong to the folklore of topology and set theory.

The aim of this note is to extend this concept to categories of compact spaces with structures, emphasizing one particular case: compact median spaces. This class was studied, in the context of \parov\ theorem, by van Mill~\cite{vanMill}, who showed that none of the obvious subspaces of the superextension of the natural numbers (that plays the role of the \v Cech-Stone compactification of $\Nat$ in the category of compact median spaces) has the \parov\ property.
Our main results are: the existence of \parov\ median spaces and an internal characterization of such spaces. We finally show that the Continuum Hypothesis is equivalent to the uniqueness of a \parov\ median space of weight $2^{\aleph_0}$, an analogue of the result of van Douwen and van Mill~\cite{vMillvDouwen}.

All topological spaces under consideration are assumed to be at least Hausdorff.
We shall use standard terminology concerning general topology and category theory.

\section{Structured compact spaces}

In this section we present category-theoretic approach to compact spaces with additional structures.
We first explain this concept informally, the precise definition is given below.
Roughly speaking, compact spaces with structures are objects of the form $\pair K a$, where $a$ is a \emph{structure} on $K$, so far with no particular meaning.
Given two compact spaces with structures $\pair K a$ and $\pair L b$, we would like to consider continuous mappings $\map f K L$ that preserve the structures.
If the structures are not clearly defined, the only thing we can say is that structure-preserving maps are \emph{some} particular continuous maps.
It is natural to require that identities should be structure-preserving and compositions of structure-preserving maps should be structure-preserving.
Here we arrive at the proper concept: Structured compact spaces can be defined by a faithful functor from some (possibly very abstract) category $\fK$ into the category of compact spaces $\komp$.

To be more precise, by $\komp$ we shall understand the category whose objects are nonempty compact spaces and arrows are continuous surjections (i.e. quotient maps).
Restricting attention to surjective maps is convenient for our applications, especially when dealing with inverse sequences, although there is no problem to consider arbitrary continuous maps.

Recall that a functor $\map F \fK \fL$ is \emph{faithful} if it is one-to-one on hom-sets, namely, given morphisms $\map f a b$, $\map g a b$, where $a,b$ are objects of $\fK$, the equality $F(f) = F(g)$ implies $f = g$.

Fix a faithful covariant functor $\map \Phi \fK \komp$.
We can look at the graph of $\Phi$ as a category whose objects are pairs of the form $\pair K a$, where $a$ is a $\fK$-object such that $K = \Phi(a)$ and an arrow from $\pair K a$ to $\pair L b$ is a $\fK$-arrow $\map f a b$, identified by $\Phi$ with a continuous surjection $\map {\Phi(f)}K L$.
The fact that $\Phi$ is faithful simply means that $f$ is uniquely determined by $\Phi(f)$.
In other words, we can call $a$ and $b$ \emph{structures} on $K$ and $L$, respectively, and say that $\Phi(f)$ is a \emph{structure-preserving} continuous map.
We shall talk about \emph{$\Phi$-structured} compact spaces, having in mind objects of the form $\pair K a$, where $a$ is a $\fK$-object such that $K = \Phi(a)$.
For simplicity, we shall say that $\map g {\pair K a}{\pair L b}$ is \emph{$\Phi$-preserving} if $g = \Phi(f)$ for a (necessarily unique) $\fK$-arrow $\map f a b$.

Summarizing, a category of structured compact spaces is given just by a faithful functor into the category of compact spaces.
This clearly covers all algebraic structures as well as many more abstract ones.

For example, consider the category $\fK$ of all Banach spaces with linear isometric embeddings. Given a Banach space $X$, let $\Phi(X)$ be the dual unit ball of $X$ endowed with the weak-star topology.
It is well-known that this is a faithful contravariant\footnote{This is not an obstacle, because one can always replace the category $\fK$ by the opposite one.} functor into $\komp$.
Thus, dual balls of Banach spaces can be viewed as compact spaces with structures, where the structure is actually a Banach space, even though all the information is ``contained" in its dual unit ball.

All the considerations above belong to the topic of concrete categories, the only formal difference is that we deal with functors into the category of compact spaces instead of the category of sets.

\subsection{\parov\ spaces}

{From} now on we fix a faithful covariant functor $\map \Phi \fK \komp$.
We shall work in the category of $\Phi$-structured compact spaces.
It is natural to define the \emph{weight} of a $\Phi$-structured compact space $\pair K a$ to be just the weight of $K$.

{From} now on, we fix a regular cardinal $\kappa \goe \aleph_0$.
We shall say that a $\Phi$-structured compact space $\pair P \pi$ is \emph{$\kappa$-\parov} if for every $\Phi$-preserving surjection $\map f {\pair L b}{\pair K a}$ of spaces of weight $<\kappa$, for every $\Phi$-preserving surjective map $\map p {\pair P \pi}{\pair K a}$ there exists a $\Phi$-preserving surjection $\map q {\pair P \pi}{\pair L b}$ such that $p =  f \cmp q$.

In order to show the existence and good properties of \parov\ spaces it is necessary to make some natural assumptions on the functor $\Phi$.

Namely, from now on we shall assume that $\Phi$ satisfies the following conditions:
\begin{enumerate}
	\item[(A)] Given $\Phi$-preserving surjections $\map f {\pair X a}{\pair Z c}$, $\map g {\pair Y b}{\pair Z c}$ with $\w(X)<\kappa$ and $\w(Y)<\kappa$, there exist $\Phi$-preserving surjective maps $\map {f'} {\pair W d}{\pair X a}$ and $\map {g'} {\pair W d}{\pair Y b}$ such that $\w(W) < \kappa$ and $f \cmp f' = g \cmp g'$.
	\item[(B)] There exist, up to isomorphism, at most $\kappa^{<\kappa}$ many $\Phi$-structured compact spaces of weight $<\kappa$.
	\item[(C)] Given an infinite cardinal $\lam$, every inverse sequence of $\Phi$-preserving surjections between spaces of weight $<\lam$ has a limit in the category of $\Phi$-preserving surjections; if the length of this sequence is $<\cf(\lam)$ then the weight of the limit is $<\lam$.
	\item[(D)] Given an uncountable cardinal $\lam$, every $\Phi$-structured compact space of weight $\lam\goe\kappa$ is the limit of a continuous inverse sequence of length $\loe\lam$ consisting of $\Phi$-structured compact spaces of weight $<\lam$ and with $\Phi$-preserving surjections.
	\item[(F)] Given a regular cardinal $\lam > \aleph_0$, given an inverse sequence $\wek X = \seq{X_\al, p^\beta_\al, \lam}$ of $\Phi$-preserving surjective maps between spaces of weight $<\lam$, given a $\Phi$-preserving surjection $\map f {\lim \wek X} Y$, there exists $\al < \lam$ and a $\Phi$-preserving surjection $\map {f'}{X_\al}Y$ such that $f = f' \cmp p_\al$, where $p_\al$ is the canonical projection onto $X_\al$.
	\item[(T)] There exists a $\Phi$-structured compact space $\pair E e$ such that every $\Phi$-structured compact space has a $\Phi$-preserving surjection onto $E$.
\end{enumerate}
It should be clear that the identity functor of $\komp$ satisfies  conditions (A)--(T). Other natural functors, coming from algebraic structures on compact spaces, are described in Subsection~\ref{SubsERGhoea} below.
In conditions (C), (D) and (F) $\lam$ is an arbitrary infinite cardinal.
Condition (A) will be sometimes called the \emph{reversed amalgamation property}.
Condition (F) will be called the \emph{factorization property}.
Condition (T) (the existence of a weakly terminal object) is assumed for the sake of convenience only, it could be replaced by a weaker one saying that for every $\Phi$-structured compact spaces $X$, $Y$ there exists a $\Phi$-structured compact $Z$ and $\Phi$-preserving surjective maps $\map f Z X$, $\map g Z Y$.

The next fact belongs to the folklore, we sketch its proof only for completeness.

\begin{lm}\label{LemKKardinlty}
Let $K$ be a compact space of weight $\loe \kappa^{<\kappa}$ and let $L$ be a compact space of weight $<\kappa$.
Then the set of all surjective maps from $K$ onto $L$ has cardinality $\loe \kappa^{<\kappa}$.
\end{lm}

\begin{pf}
Let $Z(K)$ denote the lattice of all closed $G_\delta$ subsets of $K$.
Then $|Z(K)| \loe (\kappa^{<\kappa})^{\aleph_0} = \kappa^{<\kappa}$.
Fix a closed basis $\Bee$ in $L$ of cardinality $\mu < \kappa$ that consists of closed $G_\delta$ sets.
Every surjective map $\map f K L$ is uniquely determined by its inverse restricted to $\Bee$, which is an embedding of $\Bee$ into the lattice $Z(K)$.
Finally, $Z(K)$ has at most $(\kappa^{<\kappa})^\mu = \kappa^{<\kappa}$ subsets of cardinality $\loe \mu$.
\end{pf}

\begin{lm}\label{LemKwadratJeden}
Let $\pair K a$ be a $\Phi$-structured compact space of weight $\loe \kappa^{<\kappa}$.
Then there exists a $\Phi$-structured compact space $\pair L b$ of weight $\loe \kappa^{<\kappa}$, together with a $\Phi$-preserving surjective map $\qmap h {\pair L b}{\pair K a}$, satisfying the following condition:
\begin{enumerate}
	\item[(L)] Given $\Phi$-preserving surjections $\qmap p {\pair K a}{\pair M c}$ and $\qmap f {\pair N d}{\pair M c}$ such that $\w(N) < \kappa$, there exists a $\Phi$-preserving surjection $\qmap q {\pair L b}{\pair N d}$ such that $p \cmp h = f \cmp q$.
In other words, the diagram 
$$\xymatrix{
L \ar@{>>}[r]^q \ar@{>>}[d]_h & N \ar@{>>}[d]^f \\
K \ar@{>>}[r]_p & M
}$$
is commutative.
\end{enumerate}
\end{lm}

\begin{pf}
Let us consider the family $\Ef$ consisting of all $(M,N,q,f)$ where $M,N$ are $\Phi$-structured compact spaces, $w(N)<\kappa$ and $\qmap q K M$, $\qmap f N M$ are $\Phi$-preserving surjections.
As $w(K)\le \kappa^{<\kappa}$ and $w(N)<\kappa$ for each $(M,N,q,f)\in \Ef$, by condition (B) and Lemma~\ref{LemKKardinlty} we conclude that $|\Ef|\loe (\kappa^{<\kappa})^{<\kappa}=\kappa^{<\kappa}$.
Fix an enumeration $\Ef = \setof{(M_\alpha,N_\alpha,q_\alpha,f_\alpha)}{ \alpha<\kappa^{<\kappa}}$.

We define inductively an inverse sequence $\Es = \seq{K_\alpha, p_\alpha^\beta; \alpha<\kappa^{<\kappa}}$ of 
$\Phi$-structured compact spaces whose bonding maps are $\Phi$-preserving surjections.
We start with $K_0 = K$ and at each limit stage we use condition (C).
Given $\al<\kappa^{<\kappa}$, we require that there exists a $\Phi$-preserving surjection $\qmap{r_\alpha} {K_{\alpha+1}}{N_\alpha}$ such that the following diagram
$$\xymatrix{
K_0\ar@{>>}[d]_{q_\alpha} & K_\alpha\ar@{>>}[l]_{p^\alpha_0} & K_{\alpha+1}\ar@{>>}[l]_{p_\alpha^{\alpha+1}}\ar@{>>}[ld]^{r_\alpha}\\
M_\alpha & N_\alpha\ar@{>>}[l]^{f_\alpha} & }$$ 
commutes.
This is possible, because of the reversed amalgamation (condition (A)) for the mappings $q_\al \cmp p_0^\al$, $f_\al$.

Finally, we let $L$ to be the limit of this sequence in the category of $\Phi$-structured compact spaces, which exists by condition (C).
The same condition applied to $\lam = (\kappa^{<\kappa})^+$ says that $\w(L) \loe \kappa^{<\kappa}$.
Finally, by the construction it is clear that condition (L) holds.
\end{pf}

\begin{tw}\label{ThmtwParuf}
There exists a $\kappa$-\parov\ $\Phi$-structured compact space of weight $\loe \kappa^{<\kappa}$.
\end{tw}

\begin{pf}
We construct a $\Phi$-continuous inverse sequence of length $\kappa^+$, starting from a weakly terminal object given by condition (T).
At each successor step we use Lemma~\ref{LemKwadratJeden}, while at limit steps we use condition (C).
Finally, the factorization property of $\Phi$ (condition (F) applied to $\lam = \kappa^+$) shows that the limit of this sequence is $\kappa$-\parov.
\end{pf}

The following statement is an adaptation of well-known results in \frajon\ theory, see e.g. \cite{DrGoe92} or \cite{Kub_fra}.

\begin{tw}\label{ThmgenFraisseParow}
Let $P$ be a $\kappa$-\parov\ $\Phi$-structured compact space, where $\kappa > \aleph_0$.
Then:
\begin{enumerate}
	\item[(1)] For every $\Phi$-structured compact space $K$ of weight $\le\kappa$ there exists a $\Phi$-preserving surjection $\qmap p P K$.
	\item[(2)] If $\kappa = w(P)$ is regular then for every $\Phi$-preserving surjections $f,g\colon P\twoheadrightarrow K$, where $w(K)<\kappa$, there is a $\Phi$-isomorphism $h\colon P\to P$ such that $f=g\cmp h$.
	\item[(3)] If $\kappa$ is regular then there exists at most one (up to isomorphism) $\kappa$-\parov\ $\Phi$-structured compact space of weight $\kappa$.
\end{enumerate}
\end{tw}

\begin{proof}
(1): First, suppose that $\w(K) < \kappa$.
By (T), there are $\Phi$-preserving surjections $\qmap f P E$ and $\qmap g K E$, where $E$ is a weakly terminal $\Phi$-object specified in condition (T).
By the definition of a $\kappa$-\parov\ space, we get a $\Phi$-preserving surjection $\map q P K$ additionally satisfying $g \cmp p = f$.

Assume now that $\w(K)=\kappa$.
By (D), $K = \liminv \seq {K_\al, q^\beta_\al, \kappa}$, where the sequence is continuous, all bonding maps are $\Phi$-preserving surjections and $\w(K_\al) < \kappa$ for each $\al < \kappa$.
We may assume that $K_0 = E$.
Suppose we have constructed $\Phi$-preserving surjections $\qmap {p_\al}P {K_\al}$ for $\al < \delta$, where $\delta < \kappa$ is a fixed ordinal, such that $q^\beta_\al \cmp p_\beta = p_\al$ for every $\al < \beta < \delta$.
If $\delta$ is a limit ordinal, we use the continuity of the sequence to construct $p_\delta$.
If $\delta = \beta + 1$ then we find $p_\delta$ just using the definition of a $\kappa$-\parov\ $\Phi$-structured compact space.
Finally, the surjective map $\qmap p P K$ is the limit of the sequence $\sett{p_\al}{\al<\kappa}$.
This completes the proof of (1).

(2) and (3):
We shall prove both statements simultaneously.
Namely, assume $P$ and $Q$ are two $\Phi$-structured $\kappa$-\parov\ compact spaces of weight $\kappa$.
Assume $P = \liminv \seq{X_\al, p_\al^\beta, \kappa}$ and $Q = \liminv \seq{Y_\al, q_\al^\beta, \kappa}$, where both sequences are continuous, all bonding maps are $\Phi$-preserving surjections and all spaces $X_\al$,  $Y_\al$ have weight $< \kappa$.
Now observe that the inverse sequence $\seq{X_\al, p_\al^\beta, \kappa}$ has the following property:
\begin{enumerate}
	\item[(P)] Given $\al < \kappa$, given a $\Phi$-preserving surjection $\qmap f K {X_\al}$ with $\w(K)<\kappa$, there exist $\beta > \al$ and a $\Phi$-preserving surjection $\qmap g {X_\beta} K$ such that $$f \cmp g = p^\beta_\al.$$
\end{enumerate}
This is easily proved by applying the $\kappa$-\parov\ property and the factorization property (F).
Obviously, the sequence $\seq{Y_\al, q_\al^\beta, \kappa}$ satisfies the same condition.

Now suppose for the moment that $X_0 = Y_0 = K$.
A standard back-and-forth argument gives the required $\Phi$-isomorphism from $P$ to $Q$.
This is visualized in the following commutative diagram:
$$\xymatrix{
X_0=K\ar@{>>}[d]_{\text{id}} & X_{\alpha_1}\ar@{>>}[l]_{p^{\alpha_1}_0} \ar@{>>}[ld]_{h_1}& &X_{\alpha_2}\ar@{>>}[ll]_{p_{\alpha_1}^{\alpha_2}} \ar@{>>}[ld]_{h_2} &\dots\ar@{>>}[l]& X_{\alpha_{\gamma+1}}\ar@{>>}[l] \ar@{>>}[ld]_{h_{\gamma+1}}&&X_{\alpha_{\gamma+2}}\ar@{>>}[ll]_{p_{\alpha_{\gamma+1}}^{\alpha_{\gamma+2}}} \ar@{>>}[ld]_{h_{\gamma+2}}& P
\ar@{--}[l] \\
Y_0=K &  &  Y_{\beta_1}\ar@{>>}[ll]^{q^{\beta_1}_{0}} \ar@{>>}[ul]^{j_1}&\dots\ar@{>>}[l]& Y_{\beta_{\gamma}}\ar@{>>}[l] &&Y_{\beta_{\gamma+1}}\ar@{>>}[ll]^{q^{\beta_{\gamma+1}}_{\beta_\gamma}} \ar@{>>}[ul]^{j_{\gamma+1}}&&Q\ar@{--}[ll]
}$$ 
The successor steps are made by applying condition (P) and the limit stages are taken care by condition (C). 

Now, assume that $K = E$ (a weakly terminal object from condition (T)).
Then the back-and-forth argument above shows that $P$ is isomorphic to $Q$, proving (3).

Finally, in order to show (2), let us assume that $P = Q$, $X_\al = Y_\al$ and $p_\al^\beta = q_\al^\beta$ in the diagram above.
Using the factorization property (F), we may assume that there are $\al_1,\beta_1<\kappa$ such that $f = p_{\al_1} \cmp p_0^{\al_1}$ and $g = q_{\beta_1} \cmp q^{\beta_1}_0$.
The same back-and-forth argument sketched in the diagram above gives the required isomorphism $h$ satisfying $f = g \cmp h$.
\end{proof}

\subsection{Topological algebras}\label{SubsERGhoea}

We now make a brief discussion of possible applications of the concepts of this section.
Namely, fix a countable first-order language $L$ consisting of function symbols (i.e. algebraic operations) and fix a class $\fK$ of compact $L$-algebras, that is, if $K \in \fK$ then $K$ is a compact Hausdorff space which is at the same time an $L$-model such that all $L$-operations in $K$ are continuous.
Let us assume that $\fK$ is closed under standard products, that is, a product of a family $\Ef \subs \fK$ is the usual product $\prod \Ef$ with the Tikhonov topology, and all $L$-operations are defined coordinate-wise.
Let us also assume that $\fK$ is stable under closed subalgebras.
In other words, if $K\in \fK$ and $K' \subs K$ is a closed subspace that is also closed under all $L$-operations, then $K' \in \fK$.

These two assumptions ensure that $\fK$ has inverse limits, that are actually inverse limits in the category of compact spaces.

In order to describe everything in the language of $\Phi$-structured compact spaces, we consider the natural forgetful functor $\Phi$ from $\fK$ into $\komp$, where $\fK$ is viewed as the category of all continuous surjections that preserve all $L$-operations.

It follows that $\fK$ (or formally the just-defined functor $\Phi$) satisfies conditions (A), (B) and (C).
In fact, (C) follows from the remarks above; (B) is a standard ``counting" argument (using the fact that the language is countable); (A) can be deduced from the fact that $\fK$ has products and is closed under substructures.
More precisely, given continuous epimorphisms of $L$-structures $\map f XZ$, $\map g YZ$, their pullback belongs to the category $\fK$.
Recall that the pullback of $\pair f g$ is a pair of maps $\pair {f'}{g'}$, where $\map {f'} W X$, $\map {g'} W Y$ are defined as follows:
$$W = \setof{\pair x y \in X \times Y}{f(x) = g(y)},$$
$$f'(x,y) = x \oraz g'(x,y) = y.$$
Condition (T) is satisfied as long as there exists an $L$-algebra $E$ in $\fK$ such that every other $K \in \fK$ has a continuous homomorphism onto $E$.
Finally, condition (D) requires an additional assumption.
Namely, we need to know that $\fK$ is closed under continuous homomorphic images.
More precisely, given $K \in \fK$, given a compact $L$-algebra $L$, if there exists a continuous epimorphism $\map f K L$, then $L \in \fK$.
Summarizing, we have:

\begin{tw}\label{Thmoergoierg}
Let $L$ be a countable first-order language consisting of function symbols only and let $\fK$ be a class of compact $L$-algebras with the following properties:
\begin{enumerate}
	\item[(1)] $\fK$ is stable under products and closed subalgebras,
	\item[(2)] $\fK$ is stable under continuous epimorphisms.
\end{enumerate}
Then $\fK$ satisfies conditions (A), (B), (C), (D) and (F).
\end{tw}

\begin{pf}
In view of the remarks above, only condition (D) requires a proof.
This is a rather standard closing-off argument, although we were unable to find it explicitly written in the literature, therefore we sketch it below.
We use the method of elementary submodels and we refer to \cite[Chapter 17]{KKLP} for more details and explanations.

Namely, given a compact space $K$ and given an elementary submodel $M$ of a big enough $H(\theta)$ (the family of sets of hereditary cardinality $<\theta$), there is a natural equivalence relation $\sim_M$ on $K$, defined by $x \sim_M y$ iff $M$ contains no continuous function $\map f K\Err$ with $f(x) \ne f(y)$.
The quotient space with respect to this relation is denoted by $K \by M$ and the quotient map is denoted by $\map {q_M}{K}{K\by M}$.
The crucial fact is that $K\by M$ is an $L$-algebra and $q_M$ is a homomorphism, whenever $M$ ``knows" the $L$-operations of $M$.
This is satisfied whenever $L \in M$ is countable.

Note that the weight of $K\by M$ is $\loe |M|$ and the Skolem-L\"owenheim theorem says that there exists a countable elementary submodel $M$ of $H(\theta)$ such that $K$ and any fixed countable set is in $M$.

Finally, we need to know that if $\theta$ is big enough so that $K \in H(\theta)$, then the family of all elementary substructures of $H(\theta)$ that contain $K$ and all the $L$-operations on $K$ is directed and closed under unions of arbitrary chains.
A suitable chain of elementary submodels of $H(\theta)$, starting with a countable one, provides an inverse sequence witnessing (D).
\end{pf}

One should admit that condition (T) is usually trivial to check, therefore it has been ignored in the statement above.

\section{Compact median spaces}

A topological space is \emph{supercompact} if it has a subbase $\Bee$ for its closed sets (called a \emph{binary subbase}) such that every linked family $\Ef\subs \Bee$ has nonempty intersection.
A family $\Ef$ is \emph{linked} if $A\cap B\nnempty$ for every $A,B\in\Ef$.
By Alexander's subbase lemma, every supercompact space is compact. A nontrivial result of Strok \& Szyma\'nski \cite{StSz} says that every compact metric space is supercompact. By the result of van Douwen \& van Mill \cite{vDvM}, every infinite supercompact space contains non-trivial convergent sequences, therefore not all compact spaces are supercompact. We say that two sets $A, B\subs X$ are \textit{screened} with $C, D$ if $A\cap D=\emptyset =B\cap C$
and $C\cup D= X$.
A closed subbase $\Ef$ is a
\textit{normal subbase} if any $S,T\in\Ef$ with $S\cap T=\emptyset$ are screened with some $S', T'\in\Ef$. 
	 A topological space $X$ which possesses a normal binary subbase is called \textit{normally supercompact}.
	For more information about supercompactness we refer to van Mill's book \cite{vanMill}.
We now discuss briefly convexity structures in median spaces.
For more details we refer to van de Vel's monograph~\cite{vdV}.




Supercompact spaces, especially normally supercompact spaces, have some nice geometric properties, in the sense of general convexity structures. Recall that an {\em interval convexity} ({\em convexity} for short) on the set $X$ is the family $\mathcal{G}$ of subsets of $X$ which satisfies the following conditions:
\begin{enumerate}
\item[(1)] $\emptyset,X\in\Gee$;
\item[(2)] if $\Gee'\subs \Gee$ then $\bigcap\Gee'\in\Gee$;
\item[(3)] if $A\subs X$ and for each $a, b\in A$ there is $C\in\Gee$ with $a,b\in C \subs A$ then $A\in\Gee$.
\end{enumerate}

Elements of $\Gee$ are called {\em convex sets}. A convex set whose complement is convex is called a {\em halfspace}. By (2) we can define the {\em convex hull} $\conv A$
of any set $A \subs X$ as the set $\bigcap\{F\in\Gee\colon A\subs F\}$. We write $[a,b]$ instead of $\conv\{a,b\}$ and call
it the {\em segment} joining $a$ and $b$. 

 

Let $X$ be a normally supercompact space with a fixed normal binary subbase $\Bee$.
The interval map $I_\Bee\colon X\times X\to\Pee(X)$ is defined by the formula:
$$I_\Bee(a,b)=\bigcap\{B\in \Bee\colon a,b\in B\}.$$
It is not hard to see that the family 
$\Gee_{I_\Bee}=\{A\subs X\colon  I_\Bee(a, b) \subs A\text{ for each }a,b\in A\}$ is a convexity containing $\mathcal{B}$
and $I_\Bee(a,b)=[a,b]$, where the segment $[a,b]$ is considered with respect to the convexity $\Gee_{I_\Bee}$. 
Since $\Bee$ is binary, $[x,y]\cap [x,z]\cap [y,z]\ne \emptyset$ for each $x,y,z\in X$.
Moreover, this intersection contains exactly one point $m(x,y,z)$, called the {\em median} of $x,y,z$. This follows from the fact that each two distinct points can be screened with two sets from $\mathcal{B}$. Since each element of $\mathcal{B}$ is convex, we see that every normally supercompact space satisfies the condition: 
\begin{center}
$CC_2$: each two distinct points can be screened with closed convex sets.
\end{center}
Let us note that the convexity $\Gee_{I_\Bee}$ is binary in the sense that each finite linked subfamily of $\Gee_{I_\Bee}$ has nonempty intersection (see \cite[Theorem 1.3.3]{vanMill}). 
So, each normally supercompact space is a compact space with a binary convexity satisfying  $CC_2$.

It turns out that the converse is true as well.
If $X$ is a compact space with a binary convexity $\Gee$ satisfying
$CC_2$ then, by compactness, the collection of all closed convex sets is a closed subbase for the topology of $X$.
Moreover, the collection of all closed convex sets is normal (see \cite[Proposition V.1.2]{WK}).
In conclusion, the following two classes topological spaces are equal: 
\begin{itemize}
	\item compact spaces with a binary convexity satisfying $CC_2$
	\item normally supercompact spaces.
	\end{itemize}
Such spaces will be called {\em compact median spaces}.

\subsection{Basic properties of compact median spaces}

In the class of compact median spaces the following counterpart of Urysohn lemma is  true (see \cite[Theorem 3.3]{vMW}):

\begin{prop}\label{vMW}
If $X$ is a compact median space and $x,y$ are distinct points of $X$ then there is a continuous map $f\colon X\to [0,1]$ such that $f(x)=0$, $f(y)=1$ and $f^{-1}(A)$ is convex for each interval $A\subs [0,1]$. 
\end{prop}

The map $f$ from the last theorem is a special case of a more general concept of  convexity preserving map.
Namely $f\colon X\to Y$, is called a {\em convexity preserving map}, if $f^{-1}(G)\in\Ef$ for each $G\in\Gee$, where $X$ and $Y$ are spaces with convexities $\Ef$ and $\Gee$ respectively.
In case of binary convexities, this is equivalent to
\begin{equation}
f([a,b]) \subs [f(a),f(b)] \text{ for every }a,b\in X.
\tag{cp}\label{Equcp}
\end{equation}
For median spaces, convexity preserving maps are precisely the median preserving maps (with the obvious meaning).
Indeed, a convexity preserving map $f$ necessarily preserves the median, because of (\ref{Equcp}).
On the other hand, if $f$ is median preserving and $x \in [a,b]$ then $m(x,a,b)=x$ and hence $f(x)=m(f(x),f(a),f(b))\in [f(a),f(b)]$, showing that $f$ satisfies (\ref{Equcp}).

{From} now on, we will consider the category $\mathfrak{CM}$ consisting of all compact median spaces and surjective continuous median preserving maps .
Such maps will be called {\em epimorphisms} and denoted $f\colon X\twoheadrightarrow Y$ for $X,Y\in\mathfrak{CM}$.
Recall that a subset $X$ of a median space $M$ is \emph{median-closed} (or \emph{median-stable}) if $m(x,y,z) \in X$ whenever $x,y,z \in X$.
It follows from Proposition~\ref{vMW} that:

\begin{tw}[cf. {\cite[Thm. 3.4]{vMW}}]\label{ThmCntrmse}
Every compact median space $X$ is isomorphic to a median-closed subset of the Tikhonov cube $[0,1]^\kappa$, where $\kappa=w(X)$. In the case of zero-dimensional spaces, Tikhonov cube can be replaced by the Cantor cube $\{0,1\}^\kappa$. 
\end{tw}

Let us mention the following consequence (cf.~\cite[Theorem 2.2]{KK}): 
 
\begin{tw}\label{ppp}
Each compact median space $X$ of weight $\kappa\goe\aleph_0$ can be represented as the limit of an inverse sequence of compact median spaces
$\seq {X_\alpha;p_\alpha^\beta; \alpha\le \beta<\kappa}$, where: 
\begin{enumerate}
\item $|X_0|=1$.
\item for a limit ordinal $\lam<\kappa$;
$X_\lam=\liminv \seq{X_\alpha:\;\alpha<\lam}$.
\item $w(X_\alpha)<w(X)$ for each $\alpha<\kappa$.
\item Each bonding map $p_\al^\beta$ is convexity preserving.
\item If $X$ is zero-dimensional then so is each $X_\alpha$ and moreover for each $\alpha<\kappa$ there exist closed convex sets
$A_\alpha,B_\alpha\subs X_\alpha$ such that $A_\alpha\cup
B_\alpha=X_\alpha$,
$X_{\alpha+1}=(A_\alpha\times\{0\})\cup(B_\alpha\times\{1\})$ and 
$p^{\alpha+1}_\alpha$ is the projection.
\end{enumerate}
\end{tw}

\begin{lm}\label{lim}
Let $K$, $L$ be compact median spaces, where $K$ is zero-dimensional of weight $\kappa \goe \aleph_0$, and let $f\colon K\to L$ be an epimorphism. 
There exists an inverse sequence of compact median spaces
${\cal S}= \seq {Y_\alpha,f^\beta_\alpha,\alpha<\beta<\kappa}$ 
such that $K=\liminv\cal S$ and 
\begin{enumerate}
	\item $Y_0=L$, $f_0=f$;
	\item for every limit ordinal $\lam<\kappa$,
$Y_\lam=\liminv \seq{Y_\alpha:\;\alpha<\lam}$;
\item for each $\alpha<\kappa$ there exist closed convex sets
$E_\alpha,F_\alpha\subs X_\alpha$ such that $E_\alpha\cup
F_\alpha=Y_\alpha$,
$Y_{\alpha+1}=(E_\alpha\times\{0\})\cup(F_\alpha\times\{1\})$, and
$f^{\alpha+1}_\alpha$ is the canonical projection.
\end{enumerate}
\end{lm}
\begin{pf} 
Let $K=\liminv \Es$ where $\Es = \seq{X_\alpha,p^\beta_\alpha,\alpha\le\beta<\kappa}$ is an inverse sequence of compact median zero-dimensional spaces
such as in the Theorem~\ref{ppp}.  
Inductively we define a map $h_{\alpha+1}:K\to Y_{\alpha+1}$ by $$
h_{\alpha+1}(x)= 
\begin{cases}(h_\alpha(x),0) & x\in p^{-1}_{\alpha+1}(A_\alpha\times\{0\})\\
                 (h_\alpha(x),1) & \mbox{otherwise.}

\end{cases}$$ and  $Y_{\alpha+1}=(E_\alpha\times\{0\})\cup(F_\alpha\times\{1\})$ and $q^{\alpha+1}_\alpha:
Y_{\alpha+1}\to Y_\alpha$ is the projection, where $E_\alpha=h_\alpha[p^{-1}_{\alpha+1}(A_\alpha\times\{0\})]$ and
$F_\alpha=h_\alpha[p^{-1}_{\alpha+1}(B_\alpha\times\{1\})].$ Let $Y_0=L$ and $h_0=f$. 
If $\alpha<\kappa$ is a limit ordinal then let $h_\alpha$ be the map induced by $\{h_\beta:\beta<\alpha\}$ and let
$Y_\beta=\liminv \seq{Y_\alpha:\;\alpha<\beta}$. Let $Y=\liminv \seq{Y_\alpha, q^{\beta}_\alpha,\kappa}$.
Since all $h_\alpha$ are epimorphisms, the induced map $h:K\to Y$ is an epimorphism. We claim that $h$ is one-to-one, i.e. it is an isomorphism.
Indeed, if $x,y\in K$ are different points, then there is $\alpha<\kappa$ and clopen sets $p^{-1}_{\alpha+1}(A_\alpha\times\{0\}),p^{-1}_{\alpha+1}(B_\alpha\times\{1\})$ such that $x\in p^{-1}_{\alpha+1}(A_\alpha\times\{0\})$ and $y\in p^{-1}_{\alpha+1}(B_\alpha\times\{1\}).$ Therefore $h(x)\in
q^{-1}_{\alpha+1}(E_\alpha\times\{0\})$ and $h(y)\in q^{-1}_{\alpha+1}(F_\alpha\times\{1\}).$
\end{pf}





\begin{lm}[{\cite[Lemma 1.1, p. 45]{WK}}]\label{baza}
Let $A, B$ be two disjoint closed convex subset of a zero-dimensional compact median space. Then there exists a clopen halfspace $H\subs X$ such that $A\cap H=\emptyset$ and $B\subs H$.
\end{lm}

\begin{lm}\label{half-base}
If $X$ is a compact median space and $\Raa$ is a normal family which is a subbase both for the closed sets and for the convexity, then any normal subbase $\Bee\subs \Raa$ for the closed sets is a subbase for the convexity. Moreover any $x\in X$ and 
closed convex $C\subs X\setminus \{x\}$ can be screened with some $A,B\in\Bee$.
\end{lm}

\begin{pf}Let $\Raa$ be  a normal subbase both for the closed sets and for the convexity and
let $\Bee\subs \Raa$ be a subbase for the closed sets. The convexity $\Cee'$ generated by the
family $\Bee$ satisfies the condition $CC_2$. We check that $\Cee\subs \Cee'$. Let $C\in \Cee$. Define $I:X^2\to X$ and $I':X^2\to X$ by $I(x,y)=\bigcap\{A\in\Cee: x,y\in A\}$ and $I'(x,y)=\bigcap\{A\in\Cee': x,y\in A\}$.
It suffices to check that if $a,b\in C$ then $I'(a,b)\subs C$.
Fix $a,b\in C$ and $x\in I'(a,b)$. For every $c,d,g\in X$ the set $I(c,d)\cap I(d,g)\cap I(c,g)$ is a singleton by $CC_2$. Hence $\{m(a,b,x)\}=I(a,b)\cap I(a,x)\cap
 I(b,x)\subs I'(a,b)\cap I'(a,x)\cap I'(b,x)=\{x\}$. Thus $x=m(a,b,x)\in I(a,b)\subs C.$
 
Let $x\in X$ and let $C\subs X\setminus \{x\}$ be a closed convex set. By $CC_2$, the intersection $\bigcap
 \{[x,c]:c\in C\}\cap C$ is a singleton, say $\{c_0\}$. Again by $CC_2$, there are $A,B\in\Bee$ such that
 $x\not\in B$ and $c_0\not\in A$ and $A\cup B=X$. Finally, $C\cap A=\emptyset$, because if $c\in C\cap A$ then
$c_0\in [x,c]\subs A$, which would be a contradiction. 
\end{pf}

\section{Parovi\v{c}enko median spaces}

Let $\kappa$ be a fixed infinite regular cardinal.
We shall now discuss $\kappa$-\parov\ median spaces, aiming at their internal characterization.

\begin{lm}\label{0-dim}
For every compact median space $K$ there exists a zero-dimensional compact median space $K_0$ with $w(K)=w(K_0)$ and an
epimorphism $f:K_0\twoheadrightarrow K$.
\end{lm}

\begin{pf}
By Theorem~\ref{ThmCntrmse}, we may assume that $K \subs [0,1]^\kappa$, where the cube is considered with the product median structure.
Consider the Cantor set $C = 2^\omega$ with the median operation induced from its standard linear ordering (not the product structure).
Let $\map h {C}{[0,1]}$ be the standard order preserving continuous surjection.
Then $h$ is an epimorphism of compact median spaces.
Furthermore, the power $\map {h^\kappa}{C^\kappa}{[0,1]^\kappa}$ is an epimorphism.
Finally, we may set $K_0 = (h^\kappa)^{-1}[K]$.
\end{pf}

\begin{lm}
Every $\kappa$-Parovi\v{c}enko median space is zero-dimensional.
\end{lm}

\begin{pf}
Let $P$ be a $\kappa$-Parovi\v cenko median space.
It is sufficient to show that two  distinct elements of $P$ can be separated by clopen sets.
Fix $a \ne b$ in $P$.
By Proposition~\ref{vMW}, there exists a continuous median preserving map $q\colon P\to[0,1]$ such that $q(a)=0$ and $q(b)=1$.
By Lemma~\ref{0-dim}, we can find a 0-dimensional compact median space $L$ and an epimorphism $\map f L {[0,1]}$.
Let $H$ be a clopen halfspace in $L$ that separates $f^{0}$ from $f^{-1}(1)$.
As $P$ is 
$\kappa$-Parovi\v{c}enko, there exists an epimorphism $g\colon P\to L$ such that $f\circ g=q$.
The set $W= g^{-1}(H)$ is clopen in $P$ and it separates $a$ from $b$.
\end{pf}

\begin{tw}
For every infinite regular cardinal $\kappa$, there exists a $\kappa$-Parovi\v{c}enko compact median space of weight $\le\kappa^{<\kappa}$.
\end{tw}

\begin{pf}
Let $\Phi$ denote the forgetful functor from the category of compact median spaces to $\komp$.
In view of Theorem \ref{Thmoergoierg}, conditions (A)--(D) and (F), (T) are satisfied, therefore the theorem follows from Theorem~\ref{ThmtwParuf}.
\end{pf}

By the theorem above combined with Theorems~\ref{ThmtwParuf} and \ref{ThmgenFraisseParow}(3), without any extra set-theoretic assumptions there exists a unique $\omega$-\parov\ compact median space $P_\omega$ of countable weight.
It is an easy exercise to check that $P_\omega$ carries the topology of the Cantor set.
On the other hand, its median structure seems to be rather complicated and not definable by any simple formula.

\subsection{Main result}

Let $P$ be a compact median space.
The collection of all clopen halfspaces in $P$ will be denoted by $H(P)$ and let $H(P)^+=H(P)\setminus\{ \emptyset \}$. 
Recall that $H(P)^+$ separates the points of $P$.
Given a family $\Aaa \subs H(P)$, it is natural to consider the following quotient space $P \by \Aaa$: Given $x, y \in P$, we define $x \sim y$ iff no member of $\Aaa$ separates $x$ from $y$.
The space $P \by \Aaa$ is formally the quotient $P \by \sim$, endowed with the quotient topology.
It is actually a median space, where the median is well-defined by the formula $m([x]_\sim, [y]_\sim, [z]_\sim) = m(x,y,z)$, $x,y,z \in P$. The canonical quotient map is obviously median preserving.
We shall use this construction in the proof below.
Note that the topological quotient space $P \by \Aaa$ can be defined as long as $\Aaa$ is a family of clopen subsets of $P$.
The fact that $\Aaa$ consists of halfspaces allows us to get a median structure on the quotient.
Note also that $P \by {H(P)} = P$.

\begin{tw}\label{characterization}
Given a compact median space $P$, the following are equivalent:
\begin{enumerate}
	\item[(a)] $P$ is a $\kappa$-Parovi\v{c}enko space.
	\item[(b)] If $f\colon P\to K$ is an epimorphism, $w(K)<\kappa$ and $E,F\subs K$ are closed convex sets such that $E\cup F=K$ then there exists  $H\in H(P)$ such that $f[H]=E$ and $f[P\setminus H]=F$.
	\item[(c)] $P$ is zero-dimensional and satisfies conditions:
       \begin{enumerate}
	            \item[(M1)] If $\Aaa,\Bee\subs H(P)$, $|\Aaa\cup\Bee|<\kappa$ and any $A\in\Aaa$ and $B\in\Bee$ are disjoint then there exists $C\in H(P)$ with  $\bigcup\Aaa\subs C$ and $\bigcup\Bee\subs P\setminus C$,
	            \item[(M2)] If $\Aaa\subs H(P)$ is linked and $|\Aaa|<\kappa$, then there exists $C\in H(P)^+$ such that $C\subs\bigcap\Aaa$,
	            \item[(M3)] for each $A\in H(P)^+$ there exists $B_0,B_1\in H(P)^+$ such that $B_0\cap B_1=\emptyset$ and $B_0\cup B_1\subs A$.
       \end{enumerate}
\end{enumerate}
\end{tw}

\begin{pf}
$(b)\Rightarrow (a)$
Let $f:K\to L$ and $q:P\to L$ be epimorphisms, where $K,L$ are compact median spaces of weight $<\kappa$.
By Lemma \ref{0-dim} we can assume that $K$ is zero-dimensional.
Using Lemma \ref{lim}, we find an inverse sequence
of compact median spaces ${\cal S} = \seq{Y_\alpha,f^\beta_\alpha,\alpha<\beta<\kappa}$
such that $K=\liminv\cal S$ and 
\begin{enumerate}
	\item $Y_0=L$;
	\item for a limit ordinal $\lam<\kappa$,
$Y_\lam=\liminv \seq{Y_\alpha:\;\alpha<\lam}$;
\item For each $\alpha<\kappa$ there exist closed convex sets
$A_\alpha,B_\alpha\subs X_\alpha$ such that $A_\alpha\cup
B_\alpha=Y_\alpha$, and 
$Y_{\alpha+1}=(A_\alpha\times\{0\})\cup(B_\alpha\times\{1\})$, and
$f^{\alpha+1}_\alpha$ is the canonical projection.
\end{enumerate}
Assume that we have already defined an epimorphism $g_\alpha\colon P\twoheadrightarrow X_\alpha$.
By condition $(b)$,  there is $H\in H(P)$ such that $f[H]=A_\alpha$ and $f[P\setminus H]=B_\alpha.$ 
Now, we define an epimorphism $g_{\alpha+1}\colon P\twoheadrightarrow Y_{\alpha+1}$ by the formula:
$$
g_{\alpha+1}(x)= 
\begin{cases}
      (g_\alpha(x),0), &\text{ if } x\in H\\
      (g_\alpha(x),1), & \text{otherwise.}

\end{cases}$$
If $\alpha<\kappa$ is a limit ordinal, let $g_\alpha$ be induced by $\{g_\beta:\beta<\alpha\}$.
Finally, the map $g:P\to K$ induced by $\{g_\beta:\beta<\kappa\}$ is an epimorphism and $f\circ g=q$.

$(c)\Rightarrow (b)$
Let $f:P\to K$ be an epimorphism with $\w(K)<\kappa$ and let $E,F\subs K$ be closed convex sets with $E\cup F=K$. Since $\w(K)=\tau<\kappa$, there exists a dense subset $D\subs E\cap F$ such that $|D|\leq \tau.$

\begin{claim}\label{cl1}
Given an open set $W\subs K$ whose complement is convex, there exists a family $\{H_\alpha:\alpha<\tau\}\subs H(P)$ such that $f^{-1}(W)=\bigcup_{\alpha<\tau} H_\alpha$, where $\tau = \w(K)$.
\end{claim}

\begin{pf}[Proof of Claim \ref{cl1}]
By \cite[Corollary V.1.3]{WK}, the collection $\Raa$ of closed halfspaces is a normal subbase both for the closed sets and for the convexity. Let  $\Bee\subs\Raa$ be a normal subbase for the closed sets such that $|\Bee|=\tau$.
By Lemma \ref{half-base}, $\Bee$ generates the convexity structure of $K$. Moreover, any $x\in X$ and closed convex set $C\subs X\setminus \{x\}$ can be screened with some $A,B\in\Bee$.  
Given $x\in W$, let  $A,B\in\Bee$ be such that $x\in A\subs W$ and $A\cup B=K.$  Hence, there exists a collection of closed halfspaces $\{U_\alpha:\alpha\in\tau\}$ such that 
$W=\bigcup \{U_\alpha:\alpha\in\tau\}$. By Lemma \ref{baza} applied to the closed convex disjoint sets $P\setminus f^{-1}(W)$,
 $f^{-1}(U_\alpha)$, there exists $H_\alpha\in H(P)$ such that $f^{-1}(U_\alpha)\subs H_\alpha  \subs
 f^{-1}(W)$.
Finally we get $f^{-1}(W)=\bigcup \{H_\alpha:\alpha\in\tau\}$.
\end{pf}

\begin{claim}\label{cl2}
Given $p\in K$, there are $H_0, H_1 \in H^+(P)$ such that $H_0\cap H_1=\emptyset$ and $H_0\cup H_1
\subs f^{-1}(p)$.
\end{claim}

\begin{pf}[Proof of Claim~\ref{cl2}]
Let $p\in P$. By Claim \ref{cl1}, there is a collection of clopen halfspaces
 $\{H_\alpha:\alpha\in\tau\}\subs H(P)$ such that $f^{-1}(p)=\bigcap \{H_\alpha:\alpha\in\tau\}$. By
condition $(M2)$, there is $G\in H^+(P)$ such that $G\subs \bigcap \{H_\alpha:\alpha\in\tau\}$.
Using condition $(M3)$, we find $H_1,H_0\in H^+(P)$ such that $H_0\cap H_1=\emptyset$ and $H_0\cup H_1
\subs G\subs f^{-1}(p).$ 
\end{pf}

By Claim \ref{cl1} there are $\{H^E_\alpha:\alpha<\tau\}\subs H(P)$ and $\{H^F_\alpha:\alpha<\tau\}\subs H(P)$ such that $f^{-1}(K\setminus
 E)=\bigcup\{H^F_\alpha:\alpha<\tau\}$ and $f^{-1}(K\setminus
 F)=\bigcup\{H^E_\alpha:\alpha<\tau\}$.
By Claim \ref{cl2} for each $d\in D$ there are 
 disjoint $H^1_d,H^0_d\in H(P)$ such that $H_d^0\cup H_d^1
\subs f^{-1}(d)$. Define $\Aaa=\{H^F_\alpha:\alpha<\tau\}\cup \{H^0_d:d\in D\}$ and 
$\Cee=\{H^E_\alpha:\alpha<\tau\}\cup \{H^1_d:d\in D\}$. Notice that $|\Aaa\cup\Cee|<\kappa$ and $A\cap
 C=\emptyset$ for every $A\in\Aaa$ and $C\in\Cee$. By condition $(M1)$, there is $H\in H(P)$ such that
 $\bigcup\Aaa\subs H$ and $\bigcup\Cee\subs P\setminus H.$ Since $f[\bigcup\Aaa],f[\bigcup\Cee]$ are dense subsets of $F$ and $E$ respectively, we get $f[H]=F$ and $f[P\setminus H]=E$.
 
$(a)\Rightarrow (c)$
We first show $(M1)$. 

Let $\Aaa,\Bee\subs H(P)$ be such that $|\Aaa\cup\Bee|<\kappa$ and $A\cap B=\emptyset$ for every $A\in\Aaa$ and $B\in\Bee$.
Let $L = P \by {(\Aaa\cup\Bee\cup \Aaa'\cup\Bee')}$, where $\Aaa'=\{P\setminus A: A\in\Aaa\}$ and $\Bee'=\{P\setminus B: B\in\Bee\}$.
Since $\Aaa\cup\Bee\cup \Aaa'\cup\Bee'\subs H(P)$ is a normal binary family, the canonical quotient mapping $q:P\to L$ is an epimorphism (i.e. it is median preserving).
Define $E=\bigcap\{L\setminus A^+:A\in\Aaa\}$ and $F=\bigcap\{L\setminus B^+:B\in\Bee\}.$ The sets $E,F$ 
are closed convex and $E\cup F=L$. Put $K=E\times\{0\}\cup F\times\{1\}$ and define an epimorphism $f:K\to L$
by $f(x,i)=x$. By condition $(a)$, there is an epimorphism $g:P\to K$ such that $q=f\circ g.$ Now we shall show that $\bigcup \Aaa\subs g^{-1}(F\times\{0\})$ and $\bigcup \Bee\subs g^{-1}(E\times\{1\}).$
Let $x\in A\in\Aaa$. Since $q(x)\in A^+\subs F\setminus E$, we get $g(x)\in F\times \{0\}$.

Next we show $(M2)$. 

Let $\Aaa\subs H(P)$ be a linked family and $|\Aaa|<\kappa$.
Put $L=P \by {(\Aaa\cup \Aaa')}$, where $\Aaa'=\{P\setminus A: A\in\Aaa\}$ and $\Aaa\cup\Aaa'$ is a normal binary family.
Let $q\colon P \twoheadrightarrow L$ be the canonical epimorphism.
Since $\Aaa$ is linked and $H(P)$ is binary we have $\bigcap\{A^+:A\in\Aaa\}\ne\emptyset$. Let $x\in \bigcap\{A^+:A\in\Aaa\}$ and $K = (L\times\{0\})\cup(\{x\}\times\{1\})$.
Define $f:K\to L$ by $f(x,i)=x$. Then $f$ is an epimorphism. By condition $(a)$ there is an epimorphism $g:P\to K$ such that $q=f\circ g$.
We shall prove that $g^{-1}(\{x\}\times\{1\})\subs \bigcap \Aaa$. Let $a\in g^{-1}(\{x\}\times\{1\})$
and suppose that $a\not\in A$ for some $A\in\Aaa$. Then $a\in P\setminus A$ and hence $q(a)\in 
(P\setminus A)^+$. This is a contradiction with $q(a)=f(g(a))=f(x,1)=x\in \bigcap\{A^+:A\in\Aaa\}$.

We now show $(M3)$. 

Let $A\in H(P)^+$. Define an epimorphism $q\colon P\twoh\{0,1\}$ by 
$$q(x)= 
\begin{cases}
      0 & \text{if }x\in A,\\
      1 & \mbox{otherwise.}

\end{cases}$$

Consider the set 
$$\lambda 3=\{\Raa: \Raa \mbox{ is a  maximal linked system in } \Pee(3)\}$$ 
equipped with the topology generated by the
family $\{A^+:A\subs 3\}$, where $A^+=\{\xi\in\lambda 3\colon A\in\xi\}$ (see~\cite[Chap.~II]{vanMill} for details). It is easy to see that
$\lambda 3=\{\xi_0,\xi_1,\xi_2,\xi_3\},$  where $\xi_i=\{A\subs\{0,1,2\}\colon i\in A\}$ for
 $i\in\{0,1,2\}$ and $\xi_3=\{\{0,1\},\{0,2\},\{1,2\},\{0,1,2\}\}$.  The space $\lambda 3$ is a compact median space with the median operation $m(x,y,z) = (x\cap y) \cup (x\cap z) \cup (y\cap z)$.
Let us define a map $f\colon \lambda 3\to\{0,1\}$ 
by  the formula:
 $$f(\xi_i)= 
\begin{cases}
      1 &\text{if } i=2,\\
      0 & \mbox{otherwise.} 
\end{cases}$$
Since $\{2\}^+=\{\xi_2\}$  is a clopen halfspace in $\lambda 3$, the map $f$ is an epimorphism.
By condition $(a)$, there is an epimorphism $g\colon P\twoh K$ such that $q=f\circ g$.
Now, we prove that $g^{-1}( \{0,1\}^+)\subs  A$. Let $a\in  g^{-1}( \{0,1\}^+)$. Then $g(a)\in\{0,1\}^+=\{\xi_0,\xi_1,\xi_3\}$. So $f(g(a))=0$, which implies that $q(a)=f(g(a))=0$.
Therefore $a\in A$, by the definition of $q$. 
It remains to note that $g^{-1}(\{0\}^+)$ and $g^{-1}(\{1\}^+)$ are clopen disjoint halfspaces contained in $A$, which completes the proof.
\end{pf}

\section{A characterization of the Continuum Hypothesis}

By Theorem~\ref{ThmgenFraisseParow}(3), if $\tau^+=2^\tau$ then there is a unique $\tau^+$-Parovi\v{c}enko space of weight $2^\tau$. We prove the converse implication by constructing two concrete examples of $\tau^+$-Parovi\v{c}enko spaces of weight
 $\le2^\tau$ which are not homeomorphic whenever $\tau^+\not =2^\tau$.

\begin{lm}\label{txt}
For every regular cardinal $\kappa$ such that $\tau^+\le\kappa\le 2^\tau$ there exists a $\tau^+$-Parovi\v{c}enko compact median space $P$ of weight $\le 2^\tau$ such that $$P=\varprojlim \seq{K_\alpha;p_\alpha^\beta;\alpha\le\beta<\kappa},$$
where 
\begin{enumerate}
\item $K_\alpha$ is a compact median space of weight $ \le 2^\tau$ for every $\alpha<\kappa$,

\item   $K_\alpha=\varprojlim \seq{K_\beta,p_\gamma^\beta;\gamma\le\beta<\alpha}$ and $p^\alpha_\beta$ is a projection from $K_\alpha$
onto $K_\beta$ for   a limit ordinal $\alpha<\kappa$, and $\beta<\alpha$,

\item $K_\alpha=K_\beta\times D^{2^\tau}$ and $p_\beta^\alpha$ is the projection on the second coordinate, for every even ordinal $\alpha=\beta+1$.
\end{enumerate}
\end{lm}

The same proof as in Lemma~\ref{LemKwadratJeden} works.

\begin{pf}
We define an inverse sequence $\seq{K_\alpha;p_\alpha^\beta;\alpha\le\beta<\kappa}$, where the spaces $K_\alpha$ are compact median of weight
 $\leq 2^\tau$ and the maps $p_\alpha^\beta$ are epimorphisms, in following way. Let $K_0=D^{2^\tau}$. Assume that $\alpha<\kappa$ and
  spaces $K_\beta$ for $\beta<\alpha$ and $p^\beta_\gamma$ for $\gamma\le\beta<\alpha$ are alredy defined. If $\alpha$ is a limit ordinal
let us define $K_\alpha=\varprojlim \seq{K_\beta,p_\gamma^\beta;\gamma\le\beta<\alpha}$ and $p^\alpha_\beta$ as a projection from $K_\alpha$
onto $K_\beta$ for $\beta<\alpha$. If $\alpha=\beta+1$ and $\alpha$ is an odd ordinal then we define $K_\alpha$ and $p_\beta^\alpha$
 as  in Lemma~\ref{LemKwadratJeden}.  If $\alpha=\beta+1$ and $\alpha$ is an even ordinal then we define $K_\alpha=K_\beta\times D^{2^\tau}$
 and $p_\beta^\alpha$ is a projection on second coordinate.

Let $P=\varprojlim \seq{K_\alpha;p_\alpha^\beta;\alpha\le\beta<2^\tau}$. The space $P$ is compact median 
and each projection $p_\alpha\colon P\to K_\alpha$ is an epimorhism. Obviously $w(P)=2^\tau$. 
Let $q\colon P\twoheadrightarrow L$ and $f\colon K\twoheadrightarrow L$ be epimorphisms, where $\w(K)\le\tau$. 
Since $\w(L)\le\tau$  and $\kappa$ is regular, there is an even ordinal  $\alpha<\kappa$ and an epimorphism $q'\colon K_{\alpha} \twoheadrightarrow L$ such that $q=q'
\circ p_{\alpha}$. By the construction, there is an epimorphism $p\colon K_{\alpha+1}\twoheadrightarrow K$ such that 
$q'\circ p_\alpha^{\alpha+1}=f\circ p$. It remains to note that
$$f\circ (p\circ p_{\alpha+1})=q'\circ p_\alpha^{\alpha+1}\circ p_{\alpha+1}=q'\circ p_\alpha=q,$$
which completes the proof that $P$ is a $\tau^+$-Parovi\v{c}enko space.
\end{pf}

\begin{tw}\label{pa-k}
For every regular cardinal $\kappa$ such that $\tau^+\le\kappa\le 2^\tau$ there  is a $\tau^+$-Parovi\v{c}enko compact median space $P$ of  weight $\le 2^\tau$ such that  $\chi(x,P)\ge \kappa$ for all $x\in P$.
\end{tw}
\begin{pf}
By Lemma \ref{txt}, there is $\tau^+$-Parovi\v{c}enko compact median space $P$ with properties (1)--(3).

Assume that there is $x\in P$ such that $\chi(x,P)=\lambda<\kappa.$ 
Without loss of the generality, we can assume that $\Bee_x\subseteq H(P)^+$ is a subbase of
size $\lambda$ at the point $x$.
As $\cf(\kappa)=\kappa$, there is an odd ordinal $\alpha<\lambda$ such that $\Bee_x=\{(p_\alpha)^{-1}(U_\gamma):
\gamma<\lambda\mbox{ and } U_\gamma\subseteq K_\alpha\mbox{ is clopen }\}$. Therefore $\{x\}=\bigcap \{(p_\alpha)^{-1}
(U_\gamma):\gamma<\lambda\mbox{ and } U_\gamma\subseteq K_\alpha\mbox{ is clopen }\}=(p_\alpha)^{-1}\left(\bigcap\{U_\gamma:\gamma<\lambda\mbox{ and } U_\gamma\subseteq K_\alpha\mbox{ is clopen }\}\right )$.
Setting $p_\alpha(x)=x_\alpha$, we get a contradiction with $(p_\alpha)^{-1}(x_\alpha)=(p_{\alpha+1})^{-1}((p^{\alpha+1}_\alpha)^{-1}(x_\alpha) )$ and 
$|(p^{\alpha+1}_\alpha)^{-1}(x_\alpha)|=2^{2^\tau}$.
\end{pf}

\begin{tw}\label{pa-1}
There exists a $\tau^+$-Parovi\v{c}enko  compact median space $P$ of weight $2^\tau$ and $\chi(x,P)=\tau^+$ for some $x\in P$.
\end{tw}

\begin{pf} 
By Lemma \ref{txt} there is a $\tau^+$-Parovi\v{c}enko compact median space $P_0$ with properties (1)--(3) for $\kappa=\tau^+$.

We define a sequence $\{W_\alpha:\alpha<\tau^+\}$ of $G_\delta$ closed convex sets of $P_0$. Let $A=\bigcap\{p_\alpha^{-1}(1):\alpha<\omega\}$.  Suppose we have already defined $\sett{W_\beta}{\beta<\alpha}$ so that if $\beta $ is a limit ordinal or  an odd ordinal  then 
 $W_\beta\subsetneq \bigcap\{W_\gamma:\gamma<\beta\}$ and if $\beta=\gamma+1$ is an even ordinal 
then $W_{\gamma+1}=(p_{\gamma+1})^{-1}(K_\gamma\times A)$.
If  $\alpha+1<\tau^+$ is an even ordinal 
then let $W_{\alpha+1}=(p_{\alpha+1})^{-1}(K_\alpha\times A)$.
If $\alpha<\tau^+$ is a limit ordinal or  an odd ordinal 
then by Theorem~\ref{characterization}~(M2),(M3) there is $W_\alpha\subsetneq \bigcap\{W_\beta:\beta<\alpha\}$ and 
$W_\alpha\in H(P)^+$. Let  $D=\bigcap \{U_\alpha:\alpha<\tau^+\mbox{ and } \alpha \mbox{ is an odd ordinal }\}$ and $P=P_0/D$.
Note that $\Int D=\emptyset$.
Suppose that $\Int D\ne\emptyset$.
There exists an odd ordinal $\alpha<\tau^+$ and a clopen set $U\subset K_\alpha$ such that $(p_\alpha)^{-1}(U)\subseteq D\subset (p_{\alpha+1})^{-1}(K_\alpha\times A)$.
Hence $(p^{\alpha+1}_\alpha)^{-1}(U)\subseteq K_\alpha\times A $, but this is a contradiction 
with $\emptyset=\Int A\subseteq D^{2^\tau}$.

We shall show that $P$ is a $0$-dimensional normally supercompact space. Let $q\colon P_0\to  P_0/D$ be a quotient map. 
Let $\Bee=\{q[H]:D\subseteq H\mbox{ or }D\cap H=\emptyset \mbox{ or }H\in H(P)^+\}$.
It is easy to show that $\Bee$ consists 
of clopen sets in $P$ and separates points, hence $\Bee$ is a subbase for $P$. If $\Pee\subseteq \Bee$ is a linked family then
$\{q^{-1}(H):H\in\Pee\}$ is the linked family.  Thus $P$ is a supercompact space. Since $\Bee$ consists of clopen sets, $P$ is a
normally supercompact space. 

Now, we shall show that $P$ satisfies conditions (M1)-(M3).

(M1): 	       
Assume that  $\Aaa,\Pee\subseteq H(P)$, $|\Aaa\cup\Pee|<\tau^+$ and any $A\in\Aaa$
and $B\in\Pee$ are disjoint. If there is $C\in\Aaa\cup\Pee$ such that  $q[D]\in C$ then it easy to see that there is 
$H\in H(P)^+$ such that $D\subseteq H$ and $\bigcup\Aaa\subseteq q[H]$ and $\bigcup\Pee\subseteq P\setminus q[H].$
If $q[D]\not\in\Aaa\cup\Pee$ then $D\cap V=\emptyset$ for $V\in\{q^{-1}(W):W\in\Aaa\cup\Pee\}$. There  exists 
an odd ordinal $\alpha<\tau^+$ such that $W_\alpha \cap V=\emptyset$ for $V\in\{q^{-1}(W):W\in\Aaa\cup\Pee\}.$
By the condition (M1) for $P$ there exists $H\in H(P)^+$ such that $\bigcup\{q^{-1}(W):W\in\Aaa\}\cup \{W_\alpha\}\subseteq H$ 
and $\bigcup\{q^{-1}(W):W\in\Pee\}\subseteq P\setminus H.$ Thus $D\subset H$.

(M2):	        
Assume that $\Aaa\subseteq H(P)$ is a linked family and $|\Aaa|<\tau^+$, then $\{q^{-1}(W):W\in\Aaa\}$ is a linked family. Hence
there exists $C\in H(P_0)^+$ such that $C\subseteq \bigcap\{q^{-1}(W):W\in\Aaa\}$. If $D\cap \bigcap\{q^{-1}(W):W\in\Aaa\}=\emptyset$
then $q[C]\in H(P)^+$, otherwise there exists an odd ordinal $\alpha<\tau^+$ such that $W_\alpha\subsetneq \bigcap\{q^{-1}(W):W\in\Aaa\}$. Therefore $\bigcap\{q^{-1}(W):W\in\Aaa\}\cap(P_0\setminus W_\alpha)\ne\emptyset$ and we can find 
$C\in H(P_0)^+$ such that $C\subseteq \bigcap\{q^{-1}(W):W\in\Aaa\}\cap(P_0\setminus W_\alpha)\ne\emptyset$, so $q[C]\in H(P)^+$
and  $q[C]\subseteq\bigcap\Aaa$.

(M3):	         
Let  $A\in H(P)^+$. Since $\Int D=\emptyset$, we have $q^{-1}(A)\ne D$. There exists an odd ordinal $\alpha<\tau^+$ such that
 $q^{-1}(A)\setminus W_\alpha\ne\emptyset$.  Since $W_\alpha\in H(P_0)^+$, there exists $B_0,B_1\in H(P_0)^+$ such that $B_0\cap B_1=\emptyset$ and $B_0\cup B_1\subseteq q^{-1}(A)\setminus W_\alpha\ne\emptyset$. It easy to see that $\chi(q[D],P)\le\tau^+$. Since the sequence  $\{W_\alpha:\alpha<\tau^+\mbox{ and } \alpha \mbox{ is an odd ordinal }\}$
  is strictly decreasing we have  $\chi(q[D],P)=\tau^+$. This completes the proof.
\end{pf}


\begin{wn}\label{frank}
If every two $\tau^+$-Parovi\v{c}enko spaces of weight $\le 2^\tau$ are homeomorphic then  $2^\tau=\tau^+$.
\end{wn}

\begin{pf}
Suppose that $2^{\tau}>\tau^+$, then by Theorem~\ref{pa-k} for $(\tau^{+})^{+}\le 2^\tau$ there exists a
$\tau^+$-Parovi\v{c}enko compact median space $P$ of weight $\le 2^\tau$ such that  $\chi(x,P)\ge (\tau^{+})^{+}$ for all $x\in P$.
By Theorem~\ref{pa-1}, there exists a $\tau^+$-Parovi\v{c}enko  compact median space $P_1$ of weight $2^\tau$ and $\chi(x,P_1)=\tau^+$ for some $x\in P_1$; hence $P$ and $P_1$ are not homeomorphic, a contradiction.
\end{pf}


\begin{thebibliography}{99}

\bibitem{BlSz} {\sc A. B{\l}aszczyk, A. Szyma\'nski}, {\it Concerning Parovi\v{c}enko's Theorem}, Bull. Acad. Pol. Sci. S\'er. Sci. Math. 28 (1980) 311--314.

\bibitem{vDvM} {\sc E. van Douwen, J. van Mill}, {\it Supercompact spaces\/}, Topology Appl. 13 (1982) 21--32.

\bibitem{vMillvDouwen} {\sc E. van Douwen, J. van Mill}, {\it Parovi\v cenko's characterization of $\beta \omega -\omega $ implies $CH$}, Proc. Amer. Math. Soc. 72  (1978) 539--541.

\bibitem{AD} {\sc A.~Dow}, {\it Saturated Boolean algebras and their Stone spaces\/}, Topology Appl. 21 (1985) 193--207.

\bibitem{DrGoe92} {\sc Droste, M.; G\"obel, R.}, {\it A categorical theorem on universal objects and its application in abelian group theory and computer science\/}, Proceedings of the International Conference on Algebra, Part 3 (Novosibirsk, 1989),  49--74, Contemp. Math., 131, Part 3, Amer. Math. Soc., Providence, RI, 1992.

\bibitem{Engelking} {\sc R. Engelking},  {\it A topological proof of \parov's characterization of $\beta {\bf N}-{\bf N}$}. Proceedings of the 1985 topology conference (Tallahassee, Fla., 1985). Topology Proc.  10  (1985),  no. 1, 47--53. 


\bibitem{KKLP} {\sc J. K\c{a}kol, W. Kubi\'s, M. L\'opez-Pellicer},
{\it Descriptive Topology in Selected Topics of Functional Analysis\/},
Developments in Mathematics, Vol. 24, Berlin, Springer, 2011.

\bibitem{WK} {\sc W. Kubi\'{s}}, {\it  Abstract Convex Structures in Topology and Set Theory\/},  Ph.D. thesis, 1999.

\bibitem{Kub_fra} {\sc W. Kubi\'s}, {\it \fra\ sequences: category-theoretic approach to universal homogeneous structures\/}, preprint, \href{http://arxiv.org/abs/0711.1683}{http://arxiv.org/abs/0711.1683}.

\bibitem{KK} {\sc W. Kubi\'{s}, A. Kucharski} {\it Convexity structures in zero-dimensional compact spaces\/}, Mathematica Panonica, 12/2 (2001) 177--183.

\bibitem{vanMill} {\sc J. van Mill}, {\it Supercompactness and Wallman
Spaces\/}, Math. Centre Tracts 85, Amsterdam 1977.

\bibitem{vMW} {\sc J. van Mill, E. Wattel}, {\it An external characterization of spaces which admit binary normal subbases\/}, Amer. J. Math. 100 (1978) 987--994.

\bibitem{Negre} {\sc S. Negrepontis}, {\it The Stone space of the saturated Boolean algebras\/}, Trans. Amer.  Math. Soc. 141 (1969) 515--527.

\bibitem{Parov} {\sc I.I. \parov}, {\it On a universal bicompactum of weight $\aleph$}, Dokl. Akad. Nauk SSSR 150 (1963) 36--39.

\bibitem{StSz} {\sc M. Strok, A. Szyma\'nski}, {\it Compact metric
spaces have binary bases\/}, Fund. Math. 89 (1975) 81--91.

\bibitem{vdV} {\sc M. van de Vel}, {\it Theory of Convex Structures\/}, North-Holland, Amsterdam 1993.

\end{thebibliography}
\end{document}